\documentclass{amsart}
\usepackage{amsmath}
  \usepackage{epsfig} 
 \usepackage[colorlinks=true]{hyperref}

\usepackage[latin1]{inputenc}                
\usepackage{amsthm,amsfonts,amsmath,amssymb} 
\usepackage{hyperref}                        
\usepackage{epsfig}                          
\usepackage{longtable}                       
\usepackage{mathrsfs}                        
\usepackage{dsfont}                          

\makeindex

\newcommand{\ud}{\mathrm{d}}
\newcommand{\e}{\mathrm{e}}
\newcommand{\CR}{\mathds{R}}
\newcommand{\CZ}{\mathds{Z}}
\newcommand{\CC}{\mathds{C}}

\newcommand{\depp}[2]{\frac{\partial\emph{$#1$}}{\partial{
\emph{$#2$}}}}
\newcommand{\dev}[2]{\displaystyle\frac{\ud\emph{$#1$}}{\ud\emph{$#2$}}}
\newcommand{\devv}[2]{\frac{\ud\emph{$#1$}}{\ud\emph{$#2$}}}
\newcommand{\fraction}{\displaystyle\frac}
\newcommand{\integral}[4]{\displaystyle\int_{\emph{$#1$}}^{\emph{$#2$}}
{\emph{$#3$}} \ \ud\emph{$#4$}}

\newcommand{\Somatorio}[2]{\displaystyle\sum\limits_{\emph{$#1$}}^
{\emph{$#2$}}}
\newcommand{\somatorio}[2]{\sum\limits_{\emph{$#1$}}^{\emph{$#2$}}}
\newcommand{\limite}[2]{\displaystyle\lim_{{\emph{$#1$}}\to{\emph{$#2$}}}}

\theoremstyle{plain}
\newtheorem{teo}{Theorem}[section]
\newtheorem{lem}{Lemma}[section]
\newtheorem{prop}{Proposition}[section]

\theoremstyle{definition}
\newtheorem{exam}{Example}[section]
{\newtheorem{rema}{Remark}[section]
\newtheorem{defi}{Definition}[section]}

\title[A Poincar\'{e} lemma in Geometric Quantisation]{A Poincar\'{e} lemma in Geometric Quantisation}
\author[Eva Miranda and Romero Solha]{}


 \email{eva.miranda@upc.edu }
 \email{romero.barbieri@upc.edu}

\thanks{Both authors are partially supported by the MINECO project GEOMETRIA ALGEBRAICA, SIMPLECTICA, ARITMETICA Y
APLICACIONES with reference: MTM2012-38122-C03-01. This research has also been partially supported by ESF network CAST, \emph{Contact and Symplectic Topology}. Romero Solha has been partially supported by Start-Up Erasmus Mundus External Cooperation Window
2009-2010 project.}

\begin{document}
\maketitle

\centerline{\scshape Eva Miranda}
\medskip
{\footnotesize
 \centerline{Departament de Matem\`{a}tica Aplicada I}
   \centerline{Universitat Polit\`{e}cnica de Catalunya, EPSEB}
   \centerline{Avinguda del Doctor Mara\~{n}\'{o}n, 44-50, 08028, Barcelona, Spain}} 

\medskip

\centerline{\scshape Romero Solha}
\medskip
{\footnotesize
 \centerline{Departament de Matem\`{a}tica Aplicada I}
   \centerline{Universitat Polit\`{e}cnica de Catalunya, ETSEIB}
   \centerline{Avinguda Diagonal 647, 08028, Barcelona, Spain}
}

%


\begin{abstract}
This article presents a Poincar\'{e} lemma for the Kostant complex, used to compute geometric quantisation, when the polarisation is given by a Lagrangian foliation defined by an integrable system with nondegenerate singularities.
\end{abstract}
\maketitle

\section{Introduction}

The aim of this article is to prove a Poincar\'{e} lemma for a complex that is used to compute geometric quantisation associated to a given real polarisation. It can be considered as the sequel of \cite{MiSolha}, in which the existence of a Poincar\'{e} lemma was investigated for a complex that computes the \emph{foliated cohomology} of a foliated manifold.

In \cite{MiSolha} we concluded that, if a foliation admits (special types of) singularities, then the foliated Poincar\'{e} lemma, which is well-known to hold for regular foliations, does no longer exist in general.  The motivating example was the foliation determined by the Hamiltonian vector fields of an integrable system with nondegenerate singularities.

The reason for us to consider this set of examples comes from the fact that integrable systems provide natural examples of real polarisations. Real polarisations show up in Geometric Quantisation, where additional data needs to be considered in order to choose a quantum representation space: the ingredients being  a complex line bundle, a Hermitian connection and a polarisation (see \cite{woodhouse} for more details on the theory of Geometric Quantisation).

The case of Geometric Quantisation with K\"{a}hler polarisation is a classical one: see \cite{guilleminsternbergnew,JHR,woodhouse} and references therein.

It was probably Kostant \cite{Sni,JHR} who first suggested defining this representation space as the cohomology with coefficients in the sheaf of flat sections of the prequantum bundle (since in general no global sections exist in the case of real polarisations).

Explicit computations of geometric quantisation for real polarisations have been done for Lagrangian fibrations \cite{Sni} and for Gelfand-Cetlin systems \cite{GuSt}. In both results, finding an explicit set of action-angle coordinates plays an important role in the calculation (quantisation is given by the set where action coordinates take integral values, these, in turn, are the Bohr-Sommerfeld leaves of the system), and the results of \cite{GuSt} connect with the field of representation theory, since  a computation of the dimension of the spaces on which a representation of a prescribed maximal weight is given is included in \cite{GuSt} ---it is important to point out here that singularities of Gelfand-Cetlin systems are excluded in this computation.

The recent article \cite{MiPe} considers the general case of polarisations which are not necessarily given by a fibration and develops tools for dealing with sheaf cohomology computations, discussing also some \emph{pathological} cases.

For real polarisations with singularities, these computations have been extended to the toric case \cite{Ha,Solha}, and hyperbolic singularities are considered in \cite{HaMi}.

One approach to compute this cohomology is \`{a} la de Rham, by finding a resolution of the sheaf of flat sections, as in \cite{Sni,JHR}. A different approach using \v{C}ech cohomology can be found in \cite{Ha} and \cite{HaMi}.

Following Kostant \cite{Sni,JHR}, a resolution for the sheaf of sections can be obtained by \emph{twisting} the sheaf computing foliated cohomology with the sheaf of flat sections. This observation allows to explicitly attack the problem of computing geometric quantisation with real regular polarisations. This complex is a fine resolution for the sheaf of flat sections because of the existence of a Poincar\'{e} lemma for the foliated cohomology complex.

When the real polarisation has singularities this recipe does not hold in general, because it is no longer true that the foliated cohomology complex admits a Poincar\'{e} lemma; even if a singular Poincar\'{e} lemma had been proved for the deformation complex of integrable systems with nondegenerate singularities \cite{MiNg}.

The purpose of this article is to prove that, notwithstanding the fact that local closed forms are not necessarily locally exact for foliated cohomology of singular foliations, we can still prove a Poincar\'{e} lemma for the twisted complex (called Kostant complex in the sequel) when the foliation is given by an integrable system with nondegenerate singularities. In particular, this allows to compute geometric quantisation by calculating the cohomology of the Kostant complex. In order to do so, we exploit the geometrical properties of this kind of singularities.

On the one hand, if the singularities are of elliptic or focus-focus type, the existence of circle actions with good properties allows to prove a Poincar\'{e} lemma \cite{Solha}. On the other hand, if the singularities are of hyperbolic type we can still prove a Poincar\'{e} lemma using a sharp analysis of Taylor flat functions. In this article we give a complete proof of a Poincar\'{e} lemma for Kostant complex when in the presence of hyperbolic singularities.

Throughout this article and otherwise stated, all the objects considered will be $C^\infty$; manifolds are real, Hausdorff, paracompact and connected; $C^\infty(V)$ denotes the set of smooth complex-valued functions over $V$; and the units are such that $\hbar=1$.


\subsection{Organisation of the article} Section 2 recalls prequantisation and some basic facts of integrable systems with singularities. In section 3 we define the geometric quantisation with real polarisations due to Kostant. Section 4 describes the Lie pseudoalgebras approach to foliated cohomology necessary to deal with the Kostant complex in the singular case. In section 5 we provide a brief summary of the results contained in \cite{MiSolha} and \cite{Solha} concerning the foliated cohomology and the Kostant complex when circle actions are taken into account. Section 6 contains a proof of a Poincar\'{e} lemma for geometric quantisation with hyperbolic singularities. Finally, in Section 7 we consider the case of higher dimensions.


\section{Prequantisation}\label{prequantum}

This section deals with some concepts needed to define a quantisation space. The first attempt was to visualise the quantum states as sections of a complex line bundle over the symplectic manifold, the so-called prequantum line bundle. The other notion described here, \emph{polarisations}, is a way to define a global distinction between momentum and position.


\subsection{Prequantum line bundle}

A symplectic manifold $(M,\omega)$ such that the de Rham class $[\omega]$ is integral is called prequantisable. A prequantum line bundle of $(M,\omega)$ is a Hermitian line bundle over $M$ with a connection, compatible with the Hermitian structure, $(L,\nabla^\omega)$ that satisfies $curv(\nabla^\omega)=-i\omega$ (the curvature of $\nabla^\omega$ is proportional to the symplectic form).\par

Any exact symplectic manifold satisfies $[\omega]=0$, in particular cotangent bundles with the canonical symplectic structure. In that case, the trivial line bundle is an example of a prequantum line bundle with a nontrivial connection.

The following theorem provides a relation between the above definitions (a proof can be found in \cite{Ko}):

\begin{teo} A symplectic manifold $(M,\omega)$ admits a prequantum line bundle $(L,\nabla^{\omega})$ if and only if it is prequantisable.
\end{teo}


\subsection{Polarisations given by nondegenerate integrable systems}

An integrable system on a symplectic manifold $(M,\omega)$ of dimension $2n$ is a set of $n$ functions, $f_1,\dots,f_n\in C^\infty(M)$, satisfying $\ud f_1\wedge\cdots\wedge\ud f_n\neq 0$ over an open dense subset of $M$ and $\{f_i,f_j\}=0$ for all $i,j$. The mapping $F=(f_1,\dots,f_n):M\to\CR^n$ is called a moment map.

The Poisson bracket is defined by $\{f,g\}=X_f(g)$, where $X_f$ is the unique vector field defined by the equation $\imath_{X_f}\omega=-\ud f$, called the Hamiltonian vector field of $f$.

The distribution generated by the Hamiltonian vector fields of the moment map, $\langle X_{f_1},\dots,X_{f_n}\rangle$, is involutive because $[X_f, X_g]=X_{\{f,g\}}$. Since $0=\{f_i,f_j\}=\omega(X_{f_i}, X_{f_j})$, the leaves of the associated (possibly singular) foliation are isotropic submanifolds and they are Lagrangian at points where the functions are functionally independent (generically Lagrangian foliation).\par

A real polarisation $\mathcal{P}$ is an integrable distribution of $TM$, in Sussmann's sense \cite{SSNN}, whose leaves are generically Lagrangian. The complexification of $\mathcal{P}$ is denoted by $P$ and will be called polarisation. From now on $(L,\nabla^{\omega})$ will be a prequantum line bundle and $P$ the complexification of a real polarisation of $(M,\omega)$.\par

There is a notion of nondegenerate singular points which was initially introduced by Eliasson \cite{eli1,eli2}.

We denote by $(x_1,y_1,\dots,x_n,y_n)$ a set of coordinates centered
at the origin of $\CR^{2n}$, and by $\omega$ the Darboux symplectic form
$\omega=\sum_{i=1}^n \ud x_i\wedge \ud y_i$ in this neighborhood.

When $p$ is a singular point for the moment map ($\ud f_i=0$ at $p$ for all $i$), since the functions $f_i$ are in involution
with respect to the Poisson bracket, the quadratic parts of the
functions $f_i$ commute, defining in this way an Abelian subalgebra
of $Q(2n,\CR)$  (the set of quadratic forms on
$2n$-variables). We
say that  these singularities  are of nondegenerate type if this subalgebra is a Cartan subalgebra.

Cartan subalgebras of $Q(2n,\CR)$ were classified by Williamson
in \cite{Wll}.

\begin{teo}[Williamson]{\label{Willi}}
  For any Cartan subalgebra $\mathcal C$ of $Q(2n,\CR)$ there is
  a symplectic system of coordinates $(x_1,y_1,\dots,x_n,y_n)$ in
  $\CR^{2n}$ and a basis $h_1,\dots,h_n$ of $\mathcal C$ such
  that each $h_i$ is one of the following:
  \begin{equation}
    \begin{array}{lcr}
      h_i = x_i^2 + y_i^2 & \textup{for }  1 \leq i \leq k_e \ , &
\textup{(elliptic)}  \\
      h_i = x_iy_i  &  \textup{for }  k_e+1 \leq i \leq k_e+k_h \ , &
      \textup{(hyperbolic)}\\
      \begin{cases}
        h_i = x_i y_i + x_{i+1} y_{i+1} , \\
        h_{i+1} = x_i y_{i+1}-x_{i+1} y_i
      \end{cases} &
      \begin{array}{c}
        \textup{for }  i = k_e+k_h+ 2j-1,\\ 1\leq j \leq k_f  \ .
      \end{array}
      &
      \textup{(focus-focus pair)}
 \end{array}
  \end{equation}
\end{teo}

Thus the number of elliptic components $k_e$, hyperbolic
components $k_h$ and focus-focus components $k_h$ is an
invariant of the algebra $\mathcal C$. The triple $(k_e,k_h,k_f)$ with  $n=k_e+k_h+2k_f$ is an invariant of the singularity and it is
called the Williamson type of $\mathcal C$.  Let $h_1,\dots,h_n$ be a Williamson basis of this
Cartan subalgebra. We denote by $X_i$ the Hamiltonian vector field of
$h_i$ with respect to $\omega$. Those vector fields form a basis of the
corresponding Cartan subalgebra of $\mathfrak{sp}(2n,\CR)$, and we say
that a vector field $X_i$ is hyperbolic (resp. elliptic) if the
corresponding function $h_i$ is so. We say that a pair of vector
fields $X_i,X_{i+1}$ defines a focus-focus pair if $X_i$ and $X_{i+1}$ are
the Hamiltonian vector fields associated to functions $h_i$ and
$h_{i+1}$ in a focus-focus pair.

In the local coordinates specified above, the vector fields $X_i$ take
the following form:

\begin{itemize}
\item $X_i$ is an elliptic vector field, \begin{equation}X_i=
  2\left(-y_i\frac{\partial}{\partial x_i} +x_i{\frac{\partial}{\partial
      y_i}}\right) \ ;\end{equation}
\item $X_i$ is a hyperbolic vector field,
  \begin{equation}X_i=-x_i\frac{\partial}{\partial x_i}+y_i{\frac{\partial}{\partial
      y_i}} \ ;\end{equation}
\item $X_i,X_{i+1}$ is a focus-focus pair,
  \begin{equation}X_i=-x_i\frac{\partial}{\partial
    x_{i}}+y_{i}\frac{\partial}{\partial
    y_i}-x_{i+1}\frac{\partial}{\partial
    x_{i+1}}+y_{i+1}\frac{\partial}{\partial y_{i+1}}\end{equation}and
  \begin{equation}X_{i+1}=-x_i\frac{\partial}{\partial
    x_{i+1}}+y_{i+1}\frac{\partial}{\partial
    y_i}+x_{i+1}\frac{\partial}{\partial
    x_i}-y_i\frac{\partial}{\partial y_{i+1}} \ .\end{equation}
\end{itemize}

Assume that $\mathcal{F}$ is a linear foliation on $\CR^{2n}$ with a rank 0 singularity at the
origin $p$ ($\ud f_i=0$ at $p$ for all $i$) of Williamson type $(k_e,k_h,k_f)$; the linear model for the foliation is then generated by the
vector fields above. It turns out that these type of singularities are symplectically linearisable and we can read off the local symplectic geometry of the foliation from the algebraic data associated to the singularity (Williamson type).

This is the content of the following symplectic linearisation result \cite{eli1,eli2,Mi},

\begin{teo}\label{thm:rank0uni}
Let $\omega$ be a symplectic form defined in a neighborhood $U$ of the
origin $p$ for which $\mathcal F$ is generically Lagrangian, then there exists
 a local diffeomorphism $\phi:(U,p)\longrightarrow (\phi(U),p)$
such that $\phi$ preserves the foliation and $\phi^*(\sum_i
\ud x_i\wedge \ud y_i)=\omega$, with $x_i,y_i$ local coordinates on
$(\phi(U),p)$.
\end{teo}

Furthermore, if $\mathcal{F}'$ is a generically Lagrangian foliation and has $\mathcal{F}$ as a linear foliation model near a point, one can symplectically linearise $\mathcal{F}'$ \cite{Mi}. This is equivalent to Eliasson's theorem \cite{eli1, eli2} in the  completely elliptic case.

There are normal forms for higher rank of $\ud F$  which have been obtained by the first author together with Nguyen Tien Zung \cite{Mi, mirandazungequiv} also in the case of singular nondegenerate compact orbits. When the rank of the singularity is greater than 0, a collection of regular vector fields is attached to it. We can then  reduce the $k$-rank case  to the $0$-rank case via a Marsden-Weinstein reduction associated to a natural Hamiltonian $\mathds{T}^k$-action given by the joint flow of the moment map $F$.


\section{Geometric Quantisation \`{a} la Kostant}\label{gqkostant}

The original idea of Geometric Quantisation is to associate a Hilbert space to a symplectic manifold via a prequantum line bundle and a
polarisation. Usually this is done using flat global sections of the line bundle, however, their existence is a nontrivial matter\footnote{Actually, Rawnsley \cite{JHR} (also \cite{Solha}, under slightly different hypotheses.) showed that the existence of a $S^1$-action may be an obstruction for the existence of nonzero global flat sections}. In case these global sections do not exist, Kostant suggested to consider higher cohomology groups, by taking cohomology with coefficients in the sheaf of flat sections, to define geometric quantisation.\par

\begin{exam}
Consider $M=\CR\times S^1$ with coordinates $(x,y)$ and $\omega=\ud x\wedge\ud y$. Take as $L$ the trivial complex line bundle with connection 1-form $\Theta=x\ud y$ with respect to the unitary section\footnote{Sections of $L$ can be represented by complex-valued functions over local trivialisations. When there is an identification between a section $s$ of the line bundle and a complex-valued function $f$, the bundle isomorphism will be omitted, for the sake of simplicity, and the equality $s=f$ will be used.} $\e^{ix}$ and  $\mathcal{P}=\langle\depp{}{y}\rangle$. Flat sections satisfy $\nabla^\omega_{\depp{}{y}} s=\ud s(\depp{}{y})-i\Theta(\depp{}{y})s=0$. Thus $s(x,y)=f(x)\e^{ixy}$, for some function $f$, and it has period $2\pi$ in $y$ if and only if $x\in\CZ$, for $S^1$ the unity circle. Thus flat sections are only well-defined for the set of points with $x\in\CZ$.
\end{exam}

Let $\mathcal{J}$ denote the space of local sections, $s$, of a prequantum line bundle, $L$, which are solutions of the following differential equation:
 \begin{equation}
\nabla^\omega_X s=0
\end{equation}

 \noindent for all local vector fields $X$ of the polarisation $P$. The space $\mathcal{J}$ has the structure of a sheaf and it is called the sheaf of flat sections.

\begin{defi}
The quantisation of $(M,\omega,L,\nabla^\omega,P)$ is given by
\begin{equation}
\mathcal{Q}(M)=\displaystyle\bigoplus_{k\geq 0} H^k(M;\mathcal{J}) \ ,
\end{equation}where $ H^k(M;\mathcal{J})$ are cohomology groups with values in the sheaf $\mathcal{J}$.
\end{defi}

\begin{rema} Even though $\mathcal{Q}(M)$ is just a vector space and a priori has no Hilbert structure, it will be called quantisation. The true quantisation of the triplet $(M,\omega,L,\nabla^{\omega},P)$ shall be the completion of the vector space $\mathcal{Q}(M)$, after a Hilbert structure is given, together with a Lie algebra homomorphism (possibly defined over a smaller set) between the Poisson algebra of $C^\infty(M)$ and operators on the Hilbert space. In spite of the problems that may exist in order to define geometric quantisation using $\mathcal{Q}(M)$, the first step is to compute this vector space.
\end{rema}


\section{Lie pseudoalgebras and the Kostant complex}

Instead of computing directly the cohomology groups $ H^k(M;\mathcal{J})$, the strategy is to present a resolution of the sheaf $\mathcal{J}$. For regular polarisations this has been done by Kostant \cite{Sni,JHR}. In the singular case this can be achieved via Lie pseudoalgebra representations. This section only recasts geometric quantisation under the language of Lie pseudoalgebras and its representations, the proof that the Kostant complex is a resolution for the sheaf is left to the remaining sections.\par

The set $C^{\infty}(M)$ is a commutative $\CC$-algebra and the polarisation induced by a integrable system $F:M\to\CR^n$ on $(M,\omega)$, $(P=\langle X_1,\dots,X_n \rangle_{C^\infty(M)}, [\cdot,\cdot]\big{|}_P)$, where $X_i$ is the Hamiltonian vector field of the $i$th component of the moment map, is both a $C^{\infty}(M)$-module and a $\CC$-Lie algebra; indeed, a Lie subalgebra of $(\mathfrak{X}(M),[\cdot,\cdot])$. The Lie algebra and $C^\infty(M)$-module structures are compatible in such a way that $(P,C^{\infty}(M),\CC)$ is an example of a Lie pseudoalgebra (see \cite{Mkz} for precise definitions and a nice account for the history and, various, names of this structure).\par

Considering $C^\infty(M)$ as a $C^\infty(M)$-module, $(P, [\cdot,\cdot]\big{|}_P)$ can be represented on $C^\infty(M)$ as vector fields acting on smooth functions.

Let $\Omega_P^k(M)$ denote  the space of multilinear maps $\mathrm{Hom}_{C^\infty(M)}(\wedge^k_{C^\infty(M)}P;\break C^\infty(M))$,
 we can write the following complex (Lie pseudoalgebra cohomology):
\begin{equation}
0\longrightarrow C^\infty_P(M)\hookrightarrow C^\infty(M)\stackrel{\ud_P}{\longrightarrow}\Omega_P^1(M)\stackrel{\ud_P}{\longrightarrow}\cdots\stackrel{\ud_P}{\longrightarrow}\Omega_P^n(M)\stackrel{\ud_P}{\longrightarrow}0 \ ,
\end{equation}
  \noindent where $C^\infty_P(M)=\mathrm{ker}(\ud_P)$ and $\ud_P$ is  the restriction of the de Rham differential to the directions of the polarization. Namely, given $\alpha\in \Omega_P^k(M)$ and $Y_1,\dots,Y_{k+1}\in P$,
the exterior derivative is defined by:
\begin{eqnarray}
\ud_P\alpha(Y_1,\dots,Y_{k+1})&=&\somatorio{i=1}{k+1}(-1)^{i+1}Y_i(\alpha(Y_1,\dots,\hat{Y_i},\dots,Y_{k+1}))  \\
&&+\somatorio{i<j}{}(-1)^{i+j}\alpha([Y_i,Y_j],Y_1,\dots,\hat{Y_i},\dots,\hat{Y_j},\dots,Y_{k+1}) \ .  \nonumber
\end{eqnarray}

The exterior derivative is a coboundary operator and the associated cohomology is denoted by $H^{\ \bullet}_P(M)$.
%

\begin{prop}
The restriction of the connection $\nabla^{\omega}$ to the polarisation, $\nabla:=\nabla^{\omega}\big{|}_P$, defines a representation of the Lie pseudoalgebra $(P,C^{\infty}(M),\CC)$ on $\Gamma(L)$.
\end{prop}
\begin{proof}
The space of sections of the prequantum line bundle $L$ is clearly a $C^{\infty}(M)$-module, and
\begin{equation}\label{eq313}
\nabla:\Gamma(L)\to\Omega_P^1(M)\otimes_{C^\infty(M)}\Gamma (L)
\end{equation}satisfies (by definition) the following property:
\begin{equation}
\nabla(fs)=\ud_Pf\otimes s + f\nabla s\ ,
\end{equation}for any $f\in C^\infty(M)$ and $s\in\Gamma(L)$.\par

If $X,Y\in P$, thinking of $\nabla$ as a linear map from $P$ to endomorphisms of $\Gamma(L)$,
\begin{equation}
\nabla_{[X,Y]}=\nabla_X\circ\nabla_Y-\nabla_Y\circ\nabla_X-curv(\nabla)(X,Y) \ .
\end{equation}But since $curv(\nabla^\omega)=-i\omega$ vanishes along $P$, $curv(\nabla)(X,Y)=0$ and $\nabla$ is a Lie algebra representation of $(P, [\cdot,\cdot]\big{|}_P)$ on $\Gamma(L)$ compatible with their $C^{\infty}(M)$-module structures.
\end{proof}

With respect to line bundle valued polarised forms, i.e., elements of ${S_P}^\bullet(L)=\displaystyle\bigoplus_{k\geq 0}S_P^k(L)$, where $S^k_P(L)=\Omega_P^k(M)\otimes_{C^\infty(M)}\Gamma (L)$, the previous proposition asserts that the degree $+1$ map $\ud^\nabla:{S_P}^\bullet(L)\to {S_P}^\bullet(L)$, defined by
\begin{equation}
\ud^\nabla(\alpha\otimes s)=\ud_P\alpha\otimes s+(-1)^{\mathrm{degree}(\alpha)}\alpha\wedge\nabla s \ ,
\end{equation}is a coboundary.

Thus, the associated Lie pseudoalgebra cohomology of this representantion,\break $H^\bullet({S_P}^\bullet(L))$, induces a complex (on the sheaf level). This complex is called the Kostant complex and is defined as,
\begin{equation}\label{eq72}
0\longrightarrow\mathcal{J}\hookrightarrow\mathcal{S}\stackrel{\nabla}{\longrightarrow}\mathcal{S}_P^1(L)\stackrel{\ud_\nabla}{\longrightarrow}\cdots\stackrel{\ud_\nabla}{\longrightarrow}\mathcal{S}_P^n(L)\stackrel{\ud_\nabla}{\longrightarrow}0 \ ,
\end{equation}with $\mathcal{S}$ denoting the sheaf of sections of the line bundle $L$ and $\mathcal{S}_P^k(L)$ the sheaves associated to $\Omega_P^k(M)\otimes_{C^\infty(M)}\Gamma (L)$.

\begin{rema} The only property of $L$ being used in this article is the existence of flat connections along $P$; any complex line bundle would do, not only a prequantum one ---the results here work if metaplectic correction is included (see \cite{woodhouse} for details about the role of metaplectic correction in Geometric Quantisation).
\end{rema}


\section{The de Rham foliated complex versus the Kostant complex}

In this section a proof of the fact that the Kostant complex is a fine resolution for the sheaf of flat sections is presented  for regular polarisations.  Other results involving nondegenerate singularities and the existence (or nonexistence) of a Poincar\'{e} lemma for the de Rham foliated complex and the Kostant complex are also discussed.

We first need a lemma regarding the curvature of a connection; its proof is provided for convenience.

\begin{lem}\label{unitarypotential2}At a submanifold $N\subset M$ where $curv(\nabla^\omega)\big{|}_{TN}=-i\ud\Theta$, there exists a unitary section such that $\Theta$ is its potential $1$-form.
\end{lem}
\begin{proof}Assuming $curv(\nabla^\omega)\big{|}_{TN}=-i\ud\Theta$, let $\mathcal{A}=\{A_j\}_{j\in I}$ be a contractible\footnote{Kostant uses this terminology (page 92 of \cite{Ko}), but it is the same as a good covering, as defined by Bott and Tu (page 42 of \cite{BT82}).} open cover of $N$ such that each $A_j$ is a local trivialisation of $L$ with unitary section $s_j$ (this can always be obtained, e.g. using a convenient cover made of balls with respect to a Riemannian metric). Each unitary section $s_j$ has $\Theta_j$ as a potential $1$-form of $\nabla^\omega$, and since
\begin{equation}
curv(\nabla^\omega)\big{|}_{TA_j}=-i\ud(\Theta\big{|}_{TA_j})=-i\ud\Theta_j \ ,
\end{equation}there exists real-valued functions $f_j\in C^\infty(A_j)$ such that $\Theta\big{|}_{TA_j}=\Theta_j-\ud f_j$. The unitary sections $r_j=\e^{-if_j}s_j$ have $\Theta\big{|}_{TA_j}$ as potential $1$-forms.

Any two sections $r_j$ and $r_k$ such that $A_j\cap A_k\neq\emptyset$ share the same potential $1$-form, and because of that, they differ by a nonzero constant function, $r_j=c_{jk}r_k$ at $A_j\cap A_k$. Trivially, $c_{jk}$ can be extended to the same constant over $A_k$, and $c_{jk}r_k$ is a section defined over $A_k$ such that its restriction to $A_j\cap A_k$ is exactly $r_j$, and it still has $\Theta\big{|}_{TA_k}$ as potential $1$-form. Hence, they can be glued together, using the glueing condition of sheaves, to a unitary section $r$ defined over $N$ and having $\Theta$ as potential $1$-form.
\end{proof}

The following result uses the foliated Poincar\'{e} lemma for regular foliations. A brief account of the foliated Poincar\'{e} lemma can be found in \cite{MiSolha}. The following result is also reproduced in \cite{MiSolha} with a slightly different proof.

\begin{lem}\label{locflat} Given a regular polarisation $P$, there exists a local unitary flat section on each point of $M$.
\end{lem}
\begin{proof}
The symplectic form is closed, $\ud\omega=0$, thus locally $\omega=\ud\theta$ and, since $P$ is Lagrangian, $\omega$ vanishes in the directions tangent to the leaves of $P$; which implies $\ud_P\Theta=0$, where $\Theta$ is the restriction of $\theta$ in the directions tangent to the leaves of  the polarisation. By the foliated Poincar\'{e} lemma, there exists a function $f$ such that $\ud_Pf=\Theta$; therefore, $\theta-\ud f$ satisfies $\ud(\theta-\ud f)=\omega$ and $\theta-\ud f$ vanishes in the directions tangent to the leaves.\par

Since at any subset of $M$ where $\omega$ is the differential of a $1$-form there exists a unitary section such that its associated potential is this particular $1$-form (lemma \ref{unitarypotential2}), one has a unitary section $s$ satisfying $\nabla^\omega_Xs=-i[\theta-\ud f](X)s$, which is a flat section: $\nabla^\omega_Xs=0$ for any $X\in P$ because $\theta-\ud f$ vanishes in the directions tangent to the leaves.
\end{proof}

As a consequence of the existence of unitary flat sections, elements of $\mathcal{S}^k_P(L)$ which are closed can be interpreted as the germ of closed polarised $k$-forms taking values in the sheaf $\mathcal{J}$: locally, in a trivialising neighbourhood of $L$ with a unitary flat section $s$, a $k$-form $\alpha\otimes s$ is closed if and only if $\ud_P\alpha=0$, because  $\ud_\nabla(\alpha\otimes s)=\ud_P\alpha\otimes s+(-1)^k\alpha\wedge\nabla s$, $s\neq 0$ and $\nabla s=0$. Therefore, together with the foliated Poincar\'{e} lemma, lemma \ref{locflat} implies the exactness of \eqref{eq72}.\par

The sheaves $\mathcal{S}^k_P(L)$ are fine: $\Gamma(L)$ and $\Omega^k_P(M)$ are modules over the ring of functions of $M$, and because of that, they admit partition of unity. Hence, via a Poincar\'{e} lemma, the abstract de Rham theorem \cite{BrD} implies \cite{Sni,JHR}:\par

\begin{teo}[Kostant]\label{fineresolution} The Kostant complex is a fine resolution of $\mathcal{J}$.
 Therefore each of its cohomology groups, $H^k({S_P}^\bullet(L))$, is isomorphic to $ H^k(M;\mathcal{J})$.
\end{teo}

It is important to notice that the proof of lemma \ref{locflat} relies on the existence of a Poincar\'{e} lemma for foliations. When the foliation is not regular such a theorem might not exist, and the proof of lemma \ref{locflat} is of no use; therefore, one needs a different method to prove that the Kostant complex is a fine resolution for the sheaf of flat sections.

This is exactly the situation for polarisations induced by nondegenerate integrable systems, for which we proved in \cite{MiSolha} that there is no Poincar\'{e} lemma for the foliated complex.

\begin{teo}[Miranda and Solha] The foliated Poincar\'{e} lemma does not hold for foliations defined by integrable systems with nondegenerate singularities of Williams\-on type $(k_e, k_h, 0)$.
\end{teo}

In \cite{MiSolha} we explicitly computed the cohomology groups in some instances ---in particular degree 1 and top degree for smooth systems and  in all the degrees for analytic ones. Thus, in order to prove a Poincar\'{e} lemma for the Kostant complex, different strategies need to be adopted. Luckily, it is possible to prove Poincar\'{e} lemmata for the Kostant complex when almost toric nondegenerate singularities are included in the picture.

The following result is contained in \cite{Solha}:

\begin{teo}[Solha]\label{solhateo} The Kostant complex is a fine resolution for $\mathcal{J}$ when $\mathcal{P}$ is given by a locally toric singular Lagrangian fibration or an almost toric fibration in dimension $4$.
\end{teo}

The proof of this theorem  (corollary 9.2, proposition 10.3 and results in subsection 10.2 of \cite{Solha}) is based on the existence of symplectic circle actions. Hyperbolic singularities do not share the same kind of symmetry as elliptic or focus-focus, i.e.: there is no natural symplectic circle action near purely hyperbolic singularities. Thus, again, the proof cannot be adapted to include hyperbolic singularities, and the next section is devoted to prove a Poincar\'{e} lemma for this remaining case.


\section{A Poincar\'{e} lemma for the Kostant complex with hyperbolic singularities}

We start this section fixing some notation for hyperbolic singularities.
Let $(M=\CR^2,\omega=\ud x\wedge\ud y)$ and $h:M\to\CR$ be a nondegenerate integrable system of hyperbolic type, i.e.: $h(x,y)=xy$. For this case, the real polarisation is $\mathcal{P}=\langle X\rangle$, with $X$ the Hamiltonian vector field $-x\depp{}{x}+y\depp{}{y}$.\par

$(M,\omega)$ is an exact symplectic manifold and the trivial line bundle is a prequantum line bundle for it: $L=\CC\times\CR^2$ with connection 1-form $\Theta=\frac{1}{2}(x\ud y-y\ud x)$ with respect to the unitary section $\e^{ih}$.

Consider a section $f\e^{ih}$ of the prequantum line bundle, the
 flat section equation can be written as,
\begin{equation}
\nabla f\e^{ih}=0 \ \Leftrightarrow \ X(f)=ihf \ .
\end{equation}This equation has been studied in \cite{HaMi}. Let us recall here  proposition 3.5 of that article.

\begin{prop}[Hamilton and Miranda]\label{fourquad}
Any flat section $s$ can be written as a collection
\begin{equation}
s_j = a_j(xy) e^{\frac{i}{2} xy \ln \vert{\frac{x}{y}}\vert}\qquad j=1,2,3,4 \ ;
\end{equation}where $a_j$ is a complex-valued smooth function of one variable, Taylor flat at $0$   \footnote{All Taylor coefficients are equal to zero.}, with domain such that $a_j(xy)$ is defined on the $j^\text{th}$ open quadrant of $\mathbb R^2$. Conversely, given four such $a_j$, they fit together to define a flat section $s$ using the formula above.
\end{prop}

 Thus (up to a different choice of sign) this implies that
\begin{equation}
f(x,y)\e^{ixy}=\left\{\begin{array}{ll}
0 & \text{if} \ x=0,y=0 \\
a_1(xy)\e^{\frac{i}{2}xy\ln\frac{y}{x}} & \text{if} \ x>0,y>0 \\
a_2(xy)\e^{\frac{i}{2}xy\ln\frac{-y}{x}} & \text{if} \ x>0,y<0 \\
a_3(xy)\e^{\frac{i}{2}xy\ln\frac{y}{-x}} & \text{if} \ x<0,y>0 \\
a_4(xy)\e^{\frac{i}{2}xy\ln\frac{y}{x}} & \text{if} \ x<0,y<0 \ ,
\end{array}\right.
\end{equation}where $a_j$ is a smooth complex-valued function of one variable (defined for $z\in[0,\infty)$ if $j=1,4$ or $z\in(-\infty,0]$ if $j=2,3$) and such that $\devv{^ka_j}{z^k}(0)=0$ for all $j$ and $k$.

The converse of proposition \ref{fourquad} guarantees that $H^0({S_P}^\bullet(L))$ is not trivial and is given by quadruples of Taylor flat smooth complex-valued functions of one variable, as above.\par

We also need the following property of polarised forms (proposition 5.1 in \cite{MiSolha}),

\begin{prop}\label{welldefineform}If $\alpha$ is a polarised $k$-form, $\alpha\in\Omega_P^k(\CR^{2n})$, then the following equality holds,
\begin{equation}
\alpha(X_{j_1},\dots,X_{j_k})\big{|}_{\Sigma_{j_1}\cup\cdots\cup\Sigma_{j_k}}=0 \ ,
\end{equation}where $\Sigma_i=\{p\in\CR^{2n} \ ; \ x_i(p)=y_i(p)=0\}$ denotes the vanishing set of a vector field of a Williamson basis $X_i$.
\end{prop}

\begin{proof} At every point $p\in \CR^{2n}$ the map $\alpha\in\Omega_P^k(\CR^{2n})$ reduces to an element of the dual of $\wedge^k{P}\big{|}_p$, which is a finite dimensional vector space. Since $X_i=0$ at $\Sigma_i$, for any $p\in\Sigma_i$ and vectors $Y_1(p),\dots,Y_{k-1}(p)\in P\big{|}_p$, the  expression
$\alpha_p(X_i(p),Y_1(p),\dots,Y_{k-1}(p))$ vanishes.
\end{proof}
For the computation of the first cohomology group the strategy is going to be close to the one used in \cite{Mi} and \cite{MiNg}: firstly, a formal solution is obtained and then a closed-form expression is given for the case of flat functions.

A $1$-form $\alpha\otimes\e^{ih}\in S_P^1(L)$ is exact if and only if there exists a $g\in C^\infty(\CR^2)$ satisfying
\begin{equation}\label{cohoequa}
\nabla g\e^{ih}=\alpha\otimes\e^{ih} \ \Leftrightarrow \ X(g)=ihg+\alpha(X) \ .
\end{equation}

We denote the Taylor series in $(x,y)$ of $g$ and $\alpha(X)$ near the origin $(0,0)\in\CR^2$ by,
\begin{equation}
\Somatorio{k,l=0}{\infty}g_{k,l}x^ky^l
\end{equation}and
\begin{equation}
\Somatorio{k,l=0}{\infty}f_{k,l}x^ky^l \ ,
\end{equation}with $f_{0,0}=0$, due to proposition \ref{welldefineform}.

The cohomological equation \eqref{cohoequa} in jets, then, reads
\begin{equation}
\Somatorio{k,l=0}{\infty}(l-k)g_{k,l}x^ky^l=\sqrt{-1}\Somatorio{k,l=0}{\infty}g_{k,l}x^{k+1}y^{l+1}+\Somatorio{k,l=0}{\infty}f_{k,l}x^ky^l \ .
\end{equation}And the following recursive relations lead to a solution,
\begin{equation}
\begin{array}{cc}
g_{0,0}=0 \ ; & \\ & \\
g_{k,k}=\sqrt{-1}f_{k+1,k+1} \ ,  & k>0 \ ;  \\ & \\
g_{0,k}=\fraction{f_{0,k}+\sqrt{-1}g_{0,k-1}}{k} \ ,  & k>0 \ ;  \\ & \\
g_{k,0}=\fraction{-f_{k,0}-\sqrt{-1}g_{k-1,0}}{k} \ ,  & k>0 \ ;  \\ & \\
g_{k,l}=\fraction{f_{k,l}+\sqrt{-1}g_{k-1,l-1}}{l-k} \ ,  & k\neq l>0 \ .  \\ & \\
\end{array}
\end{equation}

We can even write a closed-form expression for the jets,
\begin{equation}
\begin{array}{lc}
g_{0,0}=0 \ ; & \\ & \\
g_{k,k}=\sqrt{-1}f_{k+1,k+1} \ ,  & k>0 \ ;  \\ & \\
g_{0,k}=\fraction{1}{k!}\Somatorio{j=0}{k-1}(-1)^{\frac{j}{2}}(k-j-1)!f_{0,k-j} \ ,  & k>0 \ ;  \\ & \\
g_{k,0}=\fraction{1}{k!}\Somatorio{j=0}{k-1}(-1)^{\frac{j}{2}+1}(k-j-1)!f_{k-j,0} \ ,  & k>0 \ ;  \\ & \\
g_{k,l}=\Somatorio{j=0}{k-1}\fraction{(-1)^{\frac{j}{2}}}{(l-k)^{j+1}}f_{k-j,l-j}& \\ & \\\qquad\qquad+\Somatorio{j=0}{l-k-1}\fraction{(-1)^{\frac{k}{2}+\frac{j}{2}}(l-k-j-1)!}{(l-k)^k(l-k)!}f_{0,l-k-j} \ ,  & l>k>0 \ ;  \\ & \\
g_{k,l}=\Somatorio{j=0}{l-1}\fraction{(-1)^{\frac{j}{2}}}{(l-k)^{j+1}}f_{k-j,l-j}& \\ & \\\qquad\qquad+\Somatorio{j=0}{k-l-1}\fraction{(-1)^{\frac{l}{2}+\frac{j}{2}+1}(k-l-j-1)!}{(l-k)^l(k-l)!}f_{k-l-j,0} \ ,  & k>l>0 \ .
\end{array}
\end{equation}

This procedure solves the equation only formally. According to Borel's theorem \cite{borel}, there exists, up to Taylor flat functions\footnote{Observe that two smooth functions which have the same Taylor expansion at a point differ by a smooth function which has vanishing jet at all order at that point.} at the origin, a unique smooth function with such Taylor series.

Thus, we have proved the following:
\begin{lem}\label{lek} Any smooth  function $\tilde{g}$ whose Taylor series is defined by the previous recursive relations satisfies,
\begin{equation}
X(\tilde{g})-ih\tilde{g}-\alpha(X)=F \ ,
\end{equation}where $F$ is a Taylor flat function at the origin.
\end{lem}

Therefore, if it is possible to find a solution for
\begin{equation}
X(G)-ihG=F \ ,
\end{equation}such that $G$ is Taylor flat at the origin, the difference $\tilde{g}-G$ defines a smooth solution for the cohomological equation \eqref{cohoequa}.

One can solve this problem with the aid of the logarithmic function $\ln\gamma:\{(x,y)\in\CR^2 \ ; \ xy\neq 0\}\to\CR$, where $\ln\gamma(p)$ is the time  that it takes for a point in the diagonal, $\{(x,y)\in\CR^2 \ ; \ x=y\}$, to reach $p$ via the flow of $X$ (the diagonal point and $p$ lie over the same integral curve of $X$). This function is well defined for $xy\neq 0$.

\begin{lem}\label{leklek} For a given Taylor flat function $F$, a solution to the equation $X(G)-ihG=F$ is given by,
\begin{equation}
G=\integral{-\ln\gamma}{0}{\e^{-iht}F\circ\phi_t}{t} \ .
\end{equation}This solution is well defined and smooth over all points of $\CR^2$.
\end{lem}

\begin{rema} Observe that the smoothness of this formula still holds if parameters are considered in the function $F$. This observation will be needed for the higher dimensional discussion.

\end{rema}
\begin{proof}Before proving that the expression for $G$ is smooth and well defined, let us prove that $G$ solves the equation by computing $X(G)$. We first consider the composition of $G$ with the flow of $X$ at time $s$,
\begin{equation}
G\circ\phi_s=\integral{-\ln\gamma\circ\phi_s}{0}{\e^{-ith\circ\phi_s}F\circ\phi_t\circ\phi_s}{t}=\integral{-\ln\gamma\circ\phi_s}{0}{\e^{-ith}F\circ\phi_{t+s}}{t} \ .
\end{equation}The logarithmic function satisfies $\ln\gamma\circ\phi_s=\ln\gamma + s$ and $h\circ\phi_s=h$, thus, by a change of coordinates $\tau=t+s$:
\begin{equation}
G\circ\phi_s=\integral{-\ln\gamma - s+s}{s}{\e^{-ih(\tau-s)}F\circ\phi_\tau}{\tau}=\e^{ish}\integral{-\ln\gamma}{s}{\e^{-ith}F\circ\phi_t}{t} \ .
\end{equation}

Then, differentiating $G\circ\phi_s$ with respect to $s$,
\begin{equation}
\dev{}{s}G\circ\phi_s=ih\e^{ish}\integral{-\ln\gamma}{s}{\e^{-ith}F\circ\phi_t}{t}+F\circ\phi_s \ ,
\end{equation}and finally evaluating it in $s=0$ one gets,
\begin{equation}
X(G)=\dev{}{s}G\circ\phi_s\bigg{|}_{s=0}=ih\integral{-\ln\gamma}{0}{\e^{-ith}F\circ\phi_t}{t}+F=ihG+F \ .
\end{equation}

It is clear that $G$ is smooth and well defined over the points where the logarithmic function $\ln\gamma$ is well defined (the set $\{(x,y)\in\CR^2 \ ; \ xy\neq 0\}$). The idea now is to prove that it is continuous and well defined at the points where $h=0$.

For each point of $\{(x,y)\in\CR^2 \ ; \ xy\neq 0\}$,
\begin{equation}
|G|\leq\integral{-\ln\gamma}{0}{\left|\e^{-iht}F\circ\phi_t\right|}{t}=\integral{-\ln\gamma}{0}{\left|F\circ\phi_t\right|}{t}\leq\left|\ln\gamma\right|\displaystyle\max_{t\in [-\ln\gamma,0]}|F\circ\phi_t| \ .
\end{equation}

When $h$ approaches zero, $\ln\gamma$ diverges in a logarithmic fashion. It is left to understand how $\displaystyle\max_{t\in [-\ln\gamma,0]}|F\circ\phi_t|$ behaves.

At a point $p=(x,y)\in\CR^2$, the flow of the Hamiltonian vector field $X=-x\depp{}{x}+y\depp{}{y}$ is given by $\phi_t(p)=(\e^{-t}x,\e^ty)$. Let $p_0=(z,z)$ be a point of $\CR^2$ satisfying $\phi_t(p_0)=p$, then
\begin{equation}
\ln\gamma (p)=\left\{\begin{array}{lc}\frac{1}{2}\ln\fraction{y}{x} & \text{if} \ xy>0 \\ \frac{1}{2}\ln\fraction{-y}{x} & \text{if} \ xy<0\end{array}\right. \ ,
\end{equation}since
\begin{equation}
\e^{-t}z=x \ \Rightarrow \ t=\ln\fraction{z}{x}
\end{equation}and
\begin{equation}
\e^tz=y \ \Rightarrow \ t=\ln\fraction{y}{z} \ .
\end{equation}

Therefore,
\begin{equation}
\phi_{-\ln\gamma(p)}(p)= (|h(p)|^{\frac{1}{2}},|h(p)|^{\frac{1}{2}}) \ ,
\end{equation}which implies $\limite{|h|}{0}F\circ\phi_{-\ln\gamma}=0$ and it goes sufficiently fast to zero to guarantee that $G$ is continuous and vanishes at $h=0$, because the function $F$ is Taylor flat at the origin.

One can see that $G$ is actually smooth at $h=0$ by analising its differential (it is clear that the argument that follows holds for the higher order partial derivatives):
\begin{equation}
\ud G=\integral{-\ln\gamma}{0}{\left(\e^{-iht}{\phi_t}^*\circ\ud F\right)}{t} -iG\ud h +\e^{-ih\ln\gamma}F\circ\phi_{-\ln\gamma}\ud\ln\gamma \ .
\end{equation}

The first term converges to zero, as $h$ approaches to zero, by the same argument used above, the partial derivatives of a Taylor flat function are still Taylor flat by definition. The second term is continuous and well defined at $h=0$ because $G$ is and $h$ is smooth. It remains to analise the term $F\circ\phi_{-\ln\gamma}\ud\ln\gamma$. By l'H\^{o}pital's rule, $\limite{h}{0}\e^{-ih\ln\gamma}=1$, and since  $F$ is Taylor flat the following equality holds \begin{equation}\limite{h}{0}F\circ\phi_{-\ln\gamma}\ud\ln\gamma=0.\end{equation}
\end{proof}

Since the dimension of the generic leaves is 1, the only cohomology group to be checked is the first cohomology group.

 Lemmata \ref{lek} and \ref{leklek} yield the following,

\begin{teo}
The first cohomology group $H^1(S_P^{\ \bullet}(L))$ vanishes when the polarisation is given by an integrable
system on a two-dimensional manifold in a neighbourhood of a hyperbolic singularity.
\end{teo}


\section{Higher dimensions}

In this section we outline how to prove the Poincar\'{e} lemma for dimension greater than $2$. The idea is to use the  symplectic local model of integrable systems guaranteed by theorem \ref{thm:rank0uni}, that allows to do computations which entail the reduction to the 2-dimensional case (or 4-dimensional when there are focus-focus singularities).

A first strategy to prove the higher-dimensional case would be to adopt an algebraic point of view and investigate a K\"{u}nneth formula for these sheaves in the singular case. This is the point of view adopted in
 \cite{MiPe} for regular polarisations satisfying some finite-dimensional properties for the sheaf cohomology.

More concretely in \cite{MiPe} the first author of this article together with Francisco Presas proved,

\begin{teo}[Miranda and Presas]\label{thm:Kun}
There is an isomorphism,
\begin{equation}
 H^n(M_1 \times M_2; \mathcal{J}_{12}) \backsimeq \bigoplus_{p+q=n}  H^p(M_1; \mathcal{J}_1) \otimes  H^q(M_2; \mathcal{J}_2) \ ,
\end{equation}whenever the Geometric Quantisation associated to $(M_1,\mathcal{J}_1)$ has finite dimension, $M_1$ is compact  and $M_2$ admits a \emph{good} covering.

\end{teo}

When one allows singularities into the picture, proving such a formula equivalent becomes a tricky question since
these topological spaces are sometimes of infinite dimension (like
in the hyperbolic case \cite{HaMi}).

There is a more general K\"{u}nneth theory for sheaves that can be handled in some cases to give a direct proof of some interesting facts of  sheaf cohomology. K\"{u}nneth formula for sheaves has been studied by many authors (see for instance \cite{kaup}), and the notion of \lq\lq completion" of the
 topological tensor product is used (this probably dates back to  Grothendieck's thesis \cite{grothendieck}).

 In order to consider a completion of the cohomology group, a topology needs to be induced in the cohomology groups with coefficients in the sheaf of flat sections.
In the regular case this topology is quite intuitive. For foliated cohomology we consider the de Rham complex and this yields
an induced topological structure using the topology of the space of
forms and the fact that for differential forms the exterior derivative $\ud$ is continuous. In the case of general Fr\'{e}chet sheaves a similar strategy can be used (\cite{grothendieck} and \cite{bertelsonkunneth}).

 One could  try to adopt this approach to prove Poincar\'{e} lemma for the Kostant complex from the two-dimensional case\footnote{This approach may seem, a priori, quite na\"{\i}ve since for de Rham cohomology the proof of a K\"{u}nneth
 formula for compact manifolds precisely uses as a first step for its proof the local case in which the K\"{u}nneth formula prevails
 precisely because of Poincar\'{e} lemma (see \cite{BT82}). In this section we are just using this formula to motivate the construction
  that will follow, even if
 this type formula combined with a Mayer-Vietoris-like argument is extremely useful to compute
 geometric quantisation of actual compact manifolds as it was seen in \cite{MiPe}.}(in case there are no focus-focus fibres) or from the combination of $2$ and $4$-dimensional cases. We prefer to avoid this approach here and we rather provide a direct proof of this fact.

We start by considering the  pure hyperbolic case (Williamson type $(0, n,0)$) and then we sketch how to prove the general case by combining the case of several hyperbolic components with known results for toric and semitoric components.


\subsection{The case of Williamson type $(0, n,0)$}

Let $(M_1\times\cdots\times M_n=\CR^{2n},\omega=\somatorio{j=1}{n}\ud x_j\wedge\ud y_j)$ and $H=(h_1,\dots,h_n):M_1\times\cdots\times M_n\to\CR$ be a nondegenerate integrable system of Williamson type $(0,k_h=n,0)$, i.e.: $h_j(x_1,y_1,\dots,x_n,y_n)=x_jy_j$. In this case, the real polarisation is $\mathcal{P}=\langle X_1,\dots,X_n\rangle$, with $X_j$ the Hamiltonian vector field $-x_j\depp{}{x_j}+y_j\depp{}{y_j}$.

The product manifold $(M_1\times\cdots\times M_n,\omega)$ is an exact symplectic manifold and the trivial line bundle is a prequantum line bundle for it: $L=\CC\times\CR^{2n}$ with connection 1-form $\Theta=\frac{1}{2}\somatorio{j=1}{n}x_j\ud y_j-y_j\ud x_j$ with respect to the unitary section $s=\exp\left(i\somatorio{j=1}{n}h_j\right)$.

In order to have a Poincar\'{e} lemma, one needs to prove that $H^k(S^{\ \bullet}_P(L))$ are trivial for all\footnote{The cohomology group $H^0(S^{\ \bullet}_P(L))$ can also be computed by a parametric version of proposition \ref{fourquad}. Since the aim of this article is to provide a Poincar\'{e} Lemma, this simple computation is left aside.} $k\geq 1$. Let us start with $H^n(S^{\ \bullet}_P(L))$ which is the easiest case.

\begin{prop}\label{prop:hn} The top cohomology group
$H^n(S^{\ \bullet}_P(L))$ vanishes when the Willia\-mson type of the singularities is $(0, n,0)$ .
\end{prop}

\begin{proof}
Any line bundle valued polarised $n$-form, $\alpha\otimes s$, is automatically closed in dimension $2n$, and it is exact if and only if there exists a $\beta\in\Omega_P^{n-1}(M)$ such that,
\begin{equation}
\ud^\nabla(\beta\otimes s)=\alpha\otimes s \ .
\end{equation}Because $\nabla s=-i\Theta\otimes s$, the exactness of $\alpha\otimes s$ is equivalent to  finding  a solution to the equation: $\alpha=\ud_P\beta-(-1)^{n-1}i\beta\wedge\Theta$, or
\begin{eqnarray}
\alpha(X_1,\dots,X_n)&=&\Somatorio{j=1}{n}(-1)^{j+1}X_j(\beta(X_1,\dots,\hat{X_j},\dots,X_n)) \nonumber \\
&& \ \ -i\Somatorio{j=1}{n}(-1)^{j+1}h_j\beta(X_1,\dots,\hat{X_j},\dots,X_n) \ .
\end{eqnarray}

One can find a solution for this equation considering  a form $\beta$ satisfying $\imath_{X_1}\beta=0$ and solving the equations for $\beta(X_2,\dots,X_n)$, using the parametric versions of lemmata \ref{lek} and \ref{leklek}. This ends the proof of the proposition.
\end{proof}

In order to compute $H^1(S^{\ \bullet}_P(L))$, the parametric versions of lemmata \ref{lek} and \ref{leklek} are needed  in order to prove that some special properties satisfied by the data functions are preserved by the solution functions .
 More concretely,

\begin{lem}\label{lem:lemman} There exists a smooth solution $\tilde{g}\in C^{\infty}(\CR^{2n})$ of the equation
\begin{equation}
X_m(\tilde{g})-ih_m\tilde{g}=f \ ,
\end{equation}for a fixed $m$, satisfying  $X_l(\tilde{g})=ih_l\tilde{g}$ for all $l\neq m$ if and only if the  function $f\in C^{\infty}(\CR^{2n})$ fulfills the equality $X_l(f)=ih_lf$ for all $l\neq m$.
\end{lem}
\begin{proof}
The  derivatives in the equation $X_m(\tilde{g})-ih_m\tilde{g}=f$ only involve coordinates $x_m$ and $y_m$, and the remaining variables can be thought as parameters; thus, one can use lemmata \ref{lek} and \ref{leklek} to find a solution $\tilde{g}$.

If the solution $\tilde{g}$ satisfies $X_l(\tilde{g})=ih_l\tilde{g}$ (for $l\neq m$), then $f$ must satisfy $X_l(f)=ih_lf$ ---this is a necessary condition associated to the compatibility condition between the equations $X_m(\tilde{g})-ih_m\tilde{g}=f$ and $X_l(\tilde{g})=ih_l\tilde{g}$.

What remains to be proven is that this compatibility condition is also sufficient. Indeed, the expressions involved in the construction of $\tilde{g}$ (lemmata \ref{lek} and \ref{leklek}) also treat variables other than $x_m$ and $y_m$ as parameters, and it is clear $\tilde{g}$  has the same properties as   $f$ concerning the other variables: in particular, each equation $X_l(f)=ih_lf$ does not involve coordinates $x_m$ and $y_m$ and thus the equation is fulfilled by $\tilde{g}$.
\end{proof}

\begin{prop}\label{prop:h1n} The first cohomology group
$H^1(S^{\ \bullet}_P(L))$ vanishes when the Willi\-amson type of the singularity is $(0, n,0)$.
\end{prop}

\begin{proof}
A line bundle valued polarised $1$-form, $\alpha\otimes s$, is closed if and only if
\begin{equation}
X_l(\alpha(X_m))-ih_l\alpha(X_m)=X_m(\alpha(X_l))-ih_m\alpha(X_l) \ ,
\end{equation}for all pairs $l\neq m$, and it is exact if and only if there exists a smooth function $g$ such that
\begin{equation}
X_j(g)-ih_jg=\alpha(X_j) \ ,
\end{equation}for $j=1,\dots,n$.

Let us solve just the first equation using the parametric versions of lemmata \ref{lek} and \ref{leklek}, i.e there exists a function $g_1$ satisfying
\begin{equation}
X_1(g_1)-ih_1g_1=\alpha(X_1) \ .
\end{equation}The closedness of $\alpha\otimes s$ would, then, imply that
\begin{equation}
X_2\circ X_1(g_1)-iX_2(h_1g_1)=X_2(\alpha(X_1))=X_1(\alpha(X_2))-ih_1\alpha(X_2)+ih_2\alpha(X_1) \ ,
\end{equation}and because $[X_1,X_2]=X_1(h_2)=X_2(h_1)=0$ we can write:
\begin{equation}
X_1(X_2(g_1)-ih_2g_1-\alpha(X_2))=ih_1(X_2(g_1)-ih_2g_1-\alpha(X_2)) \ .
\end{equation}

Now we can apply lemma \ref{lem:lemman} to prove that there exists a function $\tilde{g}_{12}$ such that
\begin{equation}
X_1(\tilde{g}_{12})=ih_1\tilde{g}_{12}
\end{equation}and
\begin{equation}
X_2(\tilde{g}_{12})-ih_2\tilde{g}_{12}=f_{12}:=X_2(g_1)-ih_2g_1-\alpha(X_2) \ ,
\end{equation}then
\begin{equation}
X_j(g_1-\tilde{g}_{12})-ih_j(g_1-\tilde{g}_{12})=\alpha(X_j) \ ,
\end{equation}for $j=1,2$.

Now we take $g_2:=g_1-\tilde{g}_{12}$ and $f_{123}:=X_3(g_2)-ih_3g_2-\alpha(X_3)$, and consider again lemma \ref{lem:lemman}: because  $\alpha\otimes s$ is closed, the function $f_{123}$ satisfies
\begin{equation}
\left\{\begin{array}{l}
X_1(f_{123})=ih_1f_{123} \\
\ \ \ \ \ \ \ \ \ \ \ \ \ \ \ \ \ \ \ \ \ \ \ \ \ \ \ \  \\
X_2(f_{123})=ih_2f_{123}
\end{array}\right.
\end{equation}and there exists a function $\tilde{g}_{123}$ such that
\begin{equation}
\left\{\begin{array}{l}
X_1(\tilde{g}_{123})=ih_1\tilde{g}_{123} \\
X_2(\tilde{g}_{123})=ih_2\tilde{g}_{123} \ \ \ \ \ \ \ \ \ \ \ \  \\
X_3(\tilde{g}_{123})-ih_3\tilde{g}_{123}=f_{123}
\end{array}\right.
\end{equation}Thus,
\begin{equation}
X_j(g_2-\tilde{g}_{123})-ih_j(g_2-\tilde{g}_{123})=\alpha(X_j) \ ,
\end{equation}for $j=1,2,3$.

Just repeating this process, one finds a function $g=g_{n-1}-\tilde{g}_{1\cdots n}$ satisfying $\alpha\otimes s=\nabla gs$ and this finishes the proof of this proposition.
\end{proof}

Now propositions \ref{prop:hn} and \ref{prop:h1n}  entail the following result when the dimension of the manifold is $4$:

\begin{teo} There exists a Poincar\'{e} lemma for the Kostant complex in dimension $4$ for a polarisation  given by integrable system in a neighbourhood of an hyperbolic-hyperbolic singularity.
\end{teo}

\begin{rema}
Observe that together with the results of \cite{Solha}, this theorem asserts that the Kostant complex computes geometric quantisation when the polarisation is given by nondegenerate integrable systems in dimension $4$.
\end{rema}

For dimensions higher than $4$ (generic dimension of the leaves higher than $2$) we still need to compute the cohomology groups in degree $k\neq 1,n$.

Let us start considering the three dimensional case as an example to show the procedure that needs to be done in higher dimensions,

\begin{prop}\label{prop:h23} The second cohomology group
$H^2(S^{\ \bullet}_P(L))$ vanishes for purely hyperbolic singularities in dimension $6$.
\end{prop}

\begin{proof}
A line bundle valued polarised $2$-form, $\alpha\otimes s$, is closed if and only if
\begin{equation}
0=\Somatorio{j=1}{3}(-1)^{j+1}\left(X_j(\alpha(X_1,\dots,\hat{X_j},\dots,X_3))-ih_j\alpha(X_1,\dots,\hat{X_j},\dots,X_3)\right) \ ,
\end{equation}and it is exact if and only if there exists a $\beta\in\Omega_P^1(M)$ such that,
\begin{equation}
\alpha(X_l,X_m)=X_l(\beta(X_m))-ih_l\beta(X_m)-X_m(\beta(X_l))+ih_m\beta(X_l) \ ,
\end{equation}for all pairs $l\neq m$.

Taking $\beta(X_1)=0$ one can solve the $\alpha(X_1,X_2)$ equation for $\beta(X_2)$ using parametric versions of lemmata \ref{lek} and \ref{leklek}, then, plug this solution in the $\alpha(X_2,X_3)$ equation. This reduces the previouls system of equations to:
\begin{equation}
\left\{\begin{array}{l}
f_{13}=X_1(g_3)-ih_1(g_3) \\
\ \ \ \ \ \ \ \ \ \ \ \ \ \ \  \ \ \ \ \ \ \ \ \ \ \ \ \ \ \ \ \ , \\
f_{23}=X_2(g_3)-ih_2(g_3)
\end{array}\right.
\end{equation}where  the function $f_{13}=\alpha(X_1,X_3)$ and $f_{23}=\alpha(X_2,X_3)+X_3(\beta(X_2))-ih_3\beta(X_2)$ are data functions and $g_3=\beta(X_3)$ is unknown.

Using, again, parametric versions of lemmata \ref{lek} and \ref{leklek}, one can find a function $g_{13}$ satisfying $f_{13}=X_1(g_{13})-ih_1(g_{13})$.
\begin{equation}
X_1(f_{23})=X_1(\alpha(X_2,X_3))+X_1\circ X_3(\beta(X_2))-iX_1(h_3\beta(X_2)) \ ,
\end{equation}using $[X_1,X_3]=X_1(h_3)=X_3(h_1)=0$, one has
\begin{equation}
X_1(f_{23})=X_1(\alpha(X_2,X_3))+X_3(X_1(\beta(X_2)))-ih_3X_1(\beta(X_2)) \ .
\end{equation}The function $\beta(X_2)$ is solution of
\begin{equation}
\alpha(X_1,X_2)=X_1(\beta(X_2))-ih_1\beta(X_2) \ ,
\end{equation}and the fact that $\alpha\otimes s$ is closed gives
\begin{align}
&X_1(\alpha(X_2,X_3))-ih_1\alpha(X_2,X_3)\nonumber\\=&X_2(\alpha(X_1,X_3))-ih_2\alpha(X_1,X_3)-X_3(\alpha(X_1,X_2))+ih_3\alpha(X_1,X_2) \ ,
\end{align}
\begin{equation}
X_1(f_{23})-ih_1f_{23}=X_2(f_{13})-ih_2f_{13} \ ,
\end{equation}applying $X_2$ to $f_{13}=X_1(g_{13})-ih_1(g_{13})$ one has
\begin{equation}
X_1(X_2(g_{13})-ih_2g_{13}-f_{13})=ih_1(X_2(g_{13})-ih_2g_{13}-f_{13}) \ ,
\end{equation}because $[X_1,X_2]=X_1(h_2)=X_2(h_1)=0$.

Now we can apply lemma \ref{lem:lemman} to prove that there exists a function $\tilde{g}_{13}$ such that
\begin{equation}
X_1(\tilde{g}_{13})=ih_1\tilde{g}_{13}
\end{equation}and
\begin{equation}
X_2(\tilde{g}_{13})-ih_2\tilde{g}_{13}=X_2(g_{13})-ih_2g_{13}-f_{13} \ ,
\end{equation}then
\begin{equation}
X_j(g_{13}-\tilde{g}_{13})-ih_j(g_{13}-\tilde{g}_{13})=f_{j3} \ ,
\end{equation}for $j=1,2$. And this ends the proof of the proposition.
\end{proof}

%
%
%

In a similar way we can prove the more general statement,

\begin{prop}\label{prop:hkn} The $k$th cohomology group
$H^k(S^{\ \bullet}_P(L))$  with $k>0$ vanishes for singularities of Williamson type $(0,n,0)$.
\end{prop}

\subsection{The general case}

We can apply the same strategy as in the section above when there are other singularities.

For purely nonhyperbolic singularities, making use of the existence of a torus action, we can apply propositions 7.4 and 10.3 of \cite{Solha} to deduce Poincar\'{e} lemma from dimension 2 (or 4 if there are focus-focus singularities) to other dimensions. In those cases, we can even give a closed-form expression combining the explicit homotopy operators in \cite{MiSolha} for different circle actions which pairwise commute.\footnote{In \cite{Solha} the focus-focus in dimension bigger than 4 is not considered, so there is no explicit proof for the higher dimensional case containing focus-focus singularities. However, the proof follows the same argument as in the proof of proposition 10.3 in \cite{Solha}.}

For general singularities, then, one can combine the results of \cite{Solha} cited above with the results of subsection 7.1 to obtain the following theorem,

\begin{teo}The cohomology groups of the Kostant complex associated to a polarisation defined by an integrable system
in a neighbourhood of a singular nondegenerate point vanish in all degrees greater than $0$.
\end{teo}



\end{document}